\documentclass[11p,leqno]{amsart}
\usepackage{amsmath}
\usepackage{amsfonts}
\usepackage{amssymb}
\usepackage{graphicx}
\usepackage{color}
\newtheorem{theorem}{Theorem}[section]
\newtheorem*{theorem*}{Theorem}
\newtheorem{lemma}{Lemma}[section]
\newtheorem{corollary}[theorem]{Corollary}

\newtheorem{remark}[theorem]{Remark}

\def\aint{\frac{\ \ }{\ \ }{\hskip -0.4cm}\int}

\numberwithin{equation}{section}

\begin{document}
\title[Expansion modulus]{\bf Estimates on the modulus of expansion for vector fields solving nonlinear equations}

\author{ Lei Ni}

\address{Department of Mathematics\\
         University of California at San Diego\\
         La Jolla, CA 92093\\}
\email{lni@math.ucsd.edu}


\thanks{This work is partially supported by NSF grant DMS-1105549.}



\maketitle

\begin{abstract}
 By adapting methods of \cite{AC} we prove a sharp estimate on the expansion modulus of the gradient of the log of the  parabolic kernel  to the Sch\"ordinger operator with convex potential,  which improves an earlier work of Brascamp-Lieb. We also include alternate proofs to the improved log-concavity estimate, and to the fundamental gap theorem of Andrews-Clutterbuck via the  elliptic maximum principle. Some applications of the estimates are also obtained, including a sharp lower bound on the first eigenvalue.
\end{abstract}

\section{Introduction} In \cite{BL}, for any bounded convex domain $\Omega\subset \mathbb R^n$, it was proved that  the first eigenfunction $\phi_0$, of the operator $\mathcal{L}_q=\Delta -q(x)$ with the Dirichlet boundary value, has the property that $-\log \phi_0$ is a convex function on $\Omega$, provided that the potential $q(x)$ is a convex function on $\Omega$.
See also \cite{CS, Kr, Lions, SWYY} for  alternate proofs,  as well as generalizations/applications of this important result, via the maximum principle.

Recently, this log-concavity property of the first eigenfunction was sharpened substantially into the following convexity estimate:
\begin{equation}\label{AC-est1}
-\left(\left(\nabla \log \phi_0\right)(y)-\left(\nabla\log \phi_0\right)(x)\right)\cdot \frac{y-x}{|y-x|} \ge 2\frac{\pi}{D}
 \tan{\left(\frac{\pi |y-x|}{2D}\right)}
\end{equation}
for any pair of points $(x, y)$ in $\Omega$ with $x\ne y$, by Andrews and Clutterbuck \cite{AC}. Here $D$ is the diameter $\operatorname{Diam}(\Omega)$.
This improved convexity (or log-concavity for $\phi_0$) estimate, which is the novel crucial step in the proof to the fundamental gap conjecture,  was proved in \cite{AC} by the study of the precise asymptotics (as $t\to \infty$) to a parabolic equation (cf. Theorem 4.1 and Corollary 4.4 in \cite{AC}) as well as some delicate constructions of barrier functions via the Pr\"ufer transformation.

   Motivated by this result, via studying a modified parabolic equation for the vector fields we prove a  log-concavity estimate for the fundamental solution (with Dirichlet boundary data). The result sharpens a corresponding  result by Brascamp-Lieb  \cite{BL} (Theorem 6.1) on the log-concavity of the fundamental solution. {\it Precisely, if $H(x, y, t)$ is the fundamental solution of $\frac{\partial}{\partial t}-\mathcal{L}_q$ and if $\bar{H}(s,  t)$ is the fundamental solution of $\frac{\partial}{\partial t}-\frac{\partial^2}{\partial s^2}$ centered at $0$ with Dirichlet boundary condition on $[-\frac{D}{2}, \frac{D}{2}]$. Then for any $t>0$, $x\ne y$,
  \begin{equation}\label{BL-imp}
 - \left(\nabla_y \log H (z, y, t)-\nabla_x\log H(z, x, t)\right)\cdot \frac{y-x}{|y-x|}\ge -2\left(\log \bar{H}\right)'\left(\frac{|y-x|}{2}, t\right),
  \end{equation}
  provide that $q(x)$ is convex.
This estimate yields (\ref{AC-est1}) by taking $t\to \infty$. } As in \cite{AC}, even when $q(x)$ is not convex, the proof  can still give  a comparison result with the one-dimensional case. Since $\bar{H}'>0$ on $(0, \frac{D}{2})$, the estimate (\ref{BL-imp}) sharpens the log-concavity assertion of \cite{BL} on the fundamental solution. Our approach  allows a more direct  alternate argument to the estimate (\ref{AC-est1})  via an elliptic maximum principle on vector fields satisfying  nonlinear equations.

We also give an alternate argument, via the elliptic maximum principle, for the theorem of Andrews-Clutterbuck which resolves the fundamental gap conjecture asserting that the gap between the second and the first eigenvalue $\lambda_1-\lambda_0$ is no smaller than $3\frac{\pi^2}{D^2}$. In \cite{AC} the gap of the eigenvalues is used as the exponential decay rate  of the oscillation estimate on the solution to a related parabolic equation.  In our argument the gap appears more directly as the balanced coefficients for the comparison on the modulus of the continuity. The  ideas of studying the oscillation/continuity and expansion modulus  come from \cite{AC}. Our contribution is on adapting these ideas to study the vector field satisfying  elliptic equations and giving  an alternate (probably more direct but less intuitive)  argument with a different perturbation (avoiding the approximation argument involving the Pr\"ufer transformation).
Besides its originality,  the parabolic approach of \cite{AC} allows more general comparison results.

 As an application of (\ref{AC-est1}) we prove that the first eigenvalue $\lambda_0\ge n \left(\frac{\pi}{D}\right)^2+\inf_{\Omega} q$. This also follows from a corollary of (\ref{BL-imp}), which gives the sharp comparison of the decay rates of the fundamental solutions. A result (cf. Corollary \ref{Neumann}, as well as Corollary \ref{Li-Li}) generalizing Payne-Weinberger's estimate, as well as Li-Yau, Zhong-Yang's for the compact manifold with nonnegative Ricci,   on the lower bound of the second Neumann eigenvalue (for the Laplace with a drifting term)  also follows from the alternate proof quite quickly.

\section{Maximum principles}
Recall that a function $\omega(s): [0, +\infty) \to \mathbb R$, is called a modulus of the expansion for a vector field $X$ if
\begin{equation}\label{con-o}
\left( X(y)-X(x)\right)\cdot \frac{y-x}{|y-x|}\ge 2\omega\left(\frac{|y-x|}{2}\right).
\end{equation}
Under this terminology (\ref{AC-est1}) amounts to show that $\frac{\pi}{D}\tan\left(\frac{\pi}{D}s\right)$ is a
modulus of expansion for $X=-\nabla \log\phi_0$. On a Riemannian manifold, (\ref{con-o}) can be modified into a condition:
\begin{equation}\label{con-1}
X(\gamma(d))\cdot \gamma'(d)-X(\gamma(0))\cdot \gamma'(0) \ge  2\omega\left(\frac{r(y,x)}{2}\right)
\end{equation}
for any minimizing geodesics from $x$ to $y$ with $\gamma(0)=x$, $\gamma'(d)=y$, $d=r(x, y)$.

Let $X(x)$ be a $C^2$-vector field on $\Omega$. Assume that $X$ satisfies the differential equation:
\begin{equation}\label{eq-x}
\Delta X =2 \nabla_X X-V( x, X)
\end{equation}
where $V(x, p)$ is a $C^1$-vector field defined on $\Omega\times \mathbb R^n$, which we assume that it is jointly convex in the sense that $\omega(s)\equiv 0$ is an expansion modulus of $V$, namely
\begin{equation}\label{V-ass1}
\left(V(y, X(y))-V(x, X(x))\right)\cdot \frac{y-x}{|y-x|}\ge 0.
\end{equation}
Let $\psi(s): [0, \frac{D}{2}) \to \mathbb R$ be a $C^2$ function which satisfies that $\psi(0)=0$,
\begin{equation}\label{psi-ineq}
\psi'' \le- 2 \psi' \psi
\end{equation}
and  $\psi'< 0$.

\begin{theorem}\label{max-e}
Assume that $X(x)$ is a solution to (\ref{eq-x}) on $\Omega$, a bounded domain in $\mathbb R^n$ with diameter $D$. Let $\psi$ be a function defined above. Then
$$
\mathcal{C}(x, y)\doteqdot \left( X(y)-X(x)\right)\cdot \frac{y-x}{|y-x|} +2\psi\left(\frac{|y-x|}{2}\right)
$$
can not attain a negative minimum in the interior, namely for some $(x_0, y_0)$ with $x_0, y_0\in \Omega$.
\end{theorem}
\begin{proof} Argue by contradiction. Assume that at $(x_0, y_0)$, $\mathcal{C}(x, y)$ attains a negative minimum.
Clearly $x_0\ne y_0$ since $\mathcal{C}(x, x)=0$. Since for any $w_1\in T_{x_0}\mathbb R^n$ and $w_2\in T_{y_0}\mathbb R^n$, $\nabla_{w_1\oplus w_2} \mathcal{C}(x, y)|_{(x_0, y_0)}=0$, if we choose as in \cite{AC} a local orthonormal frame $\{e_i\}$ at $x_0$ such that $e_n=\frac{y_0-x_0}{|y_0-x_0|}$ and parallel translate them along the line interval joining $x_0, y_0$, it then implies that
at $(x_0, y_0)$,
\begin{eqnarray}
&\, &\nabla_{e_i} X(y) \cdot \frac{y-x}{|y-x|}=-\frac{X(y)-X(x)}{|y-x|}\cdot e_i =\nabla_{e_i} X(x) \cdot \frac{y-x}{|y-x|}, \mbox{ for } 1\le i\le n-1, \label{1st-eq1}\\
&\, & \nabla_{e_n} X(y) \cdot \frac{y-x}{|y-x|}=-\psi'\left(\frac{|y-x|}{2}\right)=\nabla_{e_n} X(x) \cdot \frac{y-x}{|y-x|}. \label{1st-eq2}
\end{eqnarray}
Let $E_i=e_i\oplus e_i \in T_{(x_0, y_0)}\mathbb R^n \times \mathbb R^n$ for $1\le i\le n-1$,  and $E_n=e_n \oplus (-e_n)$. Then the fact that $\mathcal{C}(x, y)$ attains its minimum at $(x_0, y_0)$ implies that
$$
\sum_{j=1}^n\nabla^2_{E_j E_j} \mathcal{C}|_{(x_0, y_0)}\ge 0.
$$
Direct calculation shows that at $(x_0, y_0)$
\begin{eqnarray}
0\le \nabla^2_{E_i E_i}\mathcal{C}&=& \left(\nabla^2_{e_i e_i} X(y)-\nabla^2_{e_i e_i} X(x)\right)\cdot \frac{y-x}{|y-x|}, \mbox{ for } 1\le i\le n-1, \label{2nd-ineq1}\\
0\le \nabla^2_{E_n E_n} \mathcal{C} &=& \left(\nabla^2_{e_n e_n} X(y)-\nabla^2_{e_n e_n} X(x)\right)\cdot \frac{y-x}{|y-x|} +2\psi''.
\end{eqnarray}
On the other hand using the equation (\ref{eq-x}), assumption (\ref{V-ass1}) we have that at $(x_0, y_0)$,
\begin{eqnarray*}
\sum_{j=1}^n \nabla^2_{E_j E_j} \mathcal{C} &=& \left(\Delta X(y) -\Delta X(x)\right) \cdot \frac{y-x}{|y-x|}+2\psi''\\
&\le& 2\left(\nabla_{X(y)} X(y)-\nabla_{X(x)}X(x)\right)\cdot \frac{y-x}{|y-x|} -4\psi' \psi.
\end{eqnarray*}
Now note that at $(x_0, y_0)$, using (\ref{1st-eq1}) and (\ref{1st-eq2}),
\begin{eqnarray*}
\nabla_{X(y)} X(y) \cdot \frac{y-x}{|y-x|}&=& \langle \nabla_{X(y)} X(y), e_n\rangle\\
&=& \sum_{j=1}^n \langle X(y), e_j\rangle \langle \nabla_{e_j} X(y), e_n\rangle\\
&=& -\frac{1}{|y-x|}\sum_{i=1}^{n-1} \langle X(y), e_i\rangle \langle X(y)-X(x), e_i\rangle -\psi' X(y)\cdot \frac{y-x}{|y-x|}.
\end{eqnarray*}
Combining the above two inequalities we conclude that at $(x_0, y_0)$,
\begin{eqnarray}
\sum_{j=1}^n \nabla^2_{E_j E_j} \mathcal{C}&\le& -\frac{2}{|y-x|}\sum_{i=1}^{n-1} \langle X(y)-X(x), e_i\rangle^2
-2\psi' (X(y)-X(x))\cdot \frac{y-x}{|y-x|} -4\psi'\psi \nonumber\\
&\le& -2\psi' \mathcal{C}. \label{est-fin1}
\end{eqnarray}
By assumption that $\mathcal{C}(x_0, y_0)<0$ and $\psi' <0$, estimate (\ref{est-fin1}) is contradictory to the fact that  $\sum_{j=1}^n \nabla^2_{E_j E_j} \mathcal{C}|_{(x_0, y_0)}\ge 0$.
\end{proof}

With little modification, the proof gives the same result for $V$ with non-vanishing expansion modulus $\omega(s)$. In this case $\psi$ is assumed to satisfy:
\begin{equation}\label{ass-gen1}
\psi''\le -2 \psi \psi' +\omega.
\end{equation}
Similar argument also proves the following result for a related parabolic equation, which is quite close to Theorem 4.1 of \cite{AC}.

\begin{theorem}\label{max-p} Assume that $X(x, t)$ on $\Omega \times [0, T]$ satisfies the equation:
\begin{equation}\label{x-eq2}
\left(\frac{\partial }{\partial t}-\Delta\right) X(x, t) =-2 \nabla_{X(x, t)} X(x, t)+V(x, X(x, t))
\end{equation}
with the vector field $V(x, X(x))$ having a modulus of expansion $\omega(s)$. Let $\psi(s, t)$ be a function defined on $[0, \frac{D}{2})\times [0, T]$, with $\psi(0, t))=0$ and $\psi'(s, t)< 0$, satisfies the parabolic inequality
\begin{equation}\label{ass-2}
\psi_t-\psi'' \ge 2 \psi' \psi-\omega.
\end{equation}
Then $\mathcal{C}(x, y, t)=\langle X(y, t)-X(x, t), \frac{y-x}{|y-x|}\rangle +2\psi\left(\frac{|y-x|}{2}, t\right)$ can not attain the negative minimum in the (parabolic) interior.
\end{theorem}

\section{Boundary asymptotics}
 First we show how to apply Theorem \ref{max-e} to obtain the estimate (\ref{AC-est1}) by establishing the boundary asymptotical estimates on $\mathcal{C}(x, y)$ for the case $X(x)=-\log \phi_0$. For this application we assume that $\Omega$ is $C^2$ and strictly convex, and we take
 $\psi(s)=-\frac{\pi}{D'}\tan \left(\frac{\pi}{D'}s\right)$ with $D'>D$, and $X=-\nabla \log \phi_0$. It is easy to see that $X$ satisfies (\ref{eq-x}) with $V(x)=\nabla q$. One can also check that $\psi'\left(\frac{|y-x|}{2}\right) <0$ and $\psi''=-2\psi \psi'$. The strategy is to prove that $\mathcal{C}(x, y)\ge 0$ for any $D'>D$ and then taking $D'\to D$ to obtain the estimate (\ref{AC-est1}). Clearly for $D'>D$, $\psi\left(\frac{|y-x|}{2}\right)$ is uniformly continuous on $\overline{\Omega}\times \overline{\Omega}$.

 Recall that $\phi_0$, the first eigenfunction (with the Dirichlet boundary) of $\mathcal{L}_q$, is a smooth function on $\overline{\Omega}$ such that $\phi_0(x)>0$ for any $x\in \Omega$, $\phi_0|_{\partial \Omega}=0$ and $\frac{\partial \phi_0}{\partial \nu}|_{\partial \Omega}<0$, where $\nu$ is the exterior unit normal. We assume that
 \begin{equation}\label{C2-bound}
 \|\phi_0\|_{C^2(\overline{\Omega})}\le \frac{A}{2}
 \end{equation}
 for some $A>0$.
 For any given $\epsilon >0$ we shall prove that
 $\mathcal{C}(x, y)\ge -\epsilon$ on $\Omega\times \Omega$. Note that on the diagonal $\Delta=\{(x, x)|\, x\in \Omega\}$, $\psi\left(\frac{|y-x|}{2}\right) \equiv 0$.
 Thus by the uniform continuity of $\psi\left(\frac{|y-x|}{2}\right)$, there exists $\eta>0$ such that on $\{(x, y)\in \Omega \times \Omega\, |\, |y-x|\le \eta\}$, a $\eta$-neighborhood of $\Delta$,  denoted by $\Delta_\eta$, $2\psi\ge -\epsilon$.

  Now let $\Omega_\delta=\{x|\phi_0(x)\ge \delta\}$. Also assume that  $\delta << \eta$.  We shall show that
   \begin{equation}\label{claim} \mathcal{C}(x, y)\ge 0 \mbox{ on } \partial\left(\Omega_\delta\times \Omega_\delta\right)\setminus \Delta_\eta,\end{equation} for $\delta$ sufficiently small. On $\partial \Delta_\eta$, by the log-concavity of Brascamp-Lieb (we can avoid appealing to this result, as explained in the remark below),  $\mathcal{C}(x, y)\ge 2\psi\ge -\epsilon$. Hence Theorem \ref{max-e} implies that $\mathcal{C}(x, y)\ge -\epsilon$. Below we shall show claim (\ref{claim}). This can be seen via the following considerations.

Since $|\nabla \phi_0|>0$ on $\partial \Omega$, we assume that there exists $\delta_0>0$ and $\theta_1>0$ such that
\begin{equation}\label{gradient-l}
|\nabla \phi_0|\ge \theta_1
\end{equation}
for $x\in \Omega \setminus \Omega_{\delta_0}$. This in particular implies that  for  $\delta\le \delta_0$,    $\partial \Omega_\delta$ is a smooth hypersurface. Since $\Omega$ is convex, we may choose $\delta_0$ small enough so that there exists $\theta_2>0$ such that for any $\delta\le \delta_0$, the second fundamental form $\operatorname{II} (\cdot, \cdot)$ of $\partial \Omega_\delta$
satisfies
\begin{equation}\label{convex-delta}
\operatorname{II} (\cdot, \cdot)\ge \theta_2 \operatorname{I}(\cdot, \cdot),
 \end{equation}where $\operatorname{I}(\cdot, \cdot)$ denotes the induced metric tensor on $\partial \Omega_\delta$.
On the level hypersurface $\partial \Omega_\delta$, the following formula is also well-known:
\begin{equation}\label{rel-1}
\operatorname{II} (\cdot, \cdot)=\frac{\nabla^2 \phi_0(\cdot, \cdot)}{|\nabla \phi_0|}
\end{equation}
 as symmetric tensors on $T\partial \Omega_\delta$. We also make $\delta_0\le \eta$.

Now let  $C_1=\frac{1}{2}\left(\frac{A^2}{\theta_1\theta_2}+A\right)$ and $\delta_1=\min\{\delta_0,  \frac{\theta_1^2}{4 C_1}, \frac{1}{\theta_2}\}$.
Since $\Omega$ is strictly convex, if $x\in \partial \Omega$ and $y\in \overline{\Omega}$ such that $|y-x|\ge \delta_1/A>0$ there exist $\theta_3>0$, depending only on $\delta_1/A$ and $\Omega$  such that
 \begin{equation}\label{acute}
 \langle -\nu_x, \frac{y-x}{|y-x|}\rangle \ge \theta_3.
 \end{equation}
By the continuity we may also assume that the same estimate holds if $\Omega$ is replaced by $\Omega_\delta$ for $\delta\le \delta_0$.

For $x\in \partial \Omega_{\delta/2}$, $y\in \overline{\Omega_{\delta/2}}$, let $\gamma(s)$ be the line interval joining $x$ to $y$ parametrized by the arc-length.
 Denote $\gamma'(s)$ by $W$. Along $\gamma(s)$, $W$ can split into the tangential part $W^{\operatorname{T}}$ and $W^\perp$ with respect to $T_{\gamma(s)} \Omega_{\phi_0(\gamma(s))}$ and the inter-normal $-\nu_{\gamma(s)}$. The estimate (\ref{acute}) asserts that $|W^\perp|\ge \theta_3$. We also have the estimate
\begin{eqnarray}
\nabla^2 \phi_0 (W, W) &=& \nabla^2  \phi_0 (W^{\operatorname{T}}, W^{\operatorname{T}})+2 \nabla^2  \phi_0 (W^{\operatorname{T}}, W^\perp)+ \nabla^2  \phi_0 (W^\perp, W^\perp)\nonumber\\
&\le& -|\nabla \phi_0| \operatorname{II}(W^{\operatorname{T}}, W^{\operatorname{T}})+A |W^\perp||W^{\operatorname{T}}|+\frac{A}{2}|W^\perp|^2\nonumber\\
&\le& -\theta_1\theta_2 |W^{\operatorname{T}}|^2+A |W^\perp||W^{\operatorname{T}}|+\frac{A}{2}|W^\perp|^2\nonumber\\
&\le& -\frac{\theta_1\theta_2}{2}|W^{\operatorname{T}}|^2+C_1(\theta_1, \theta_2, A)|W^\perp|^2. \label{hessian-u1}
\end{eqnarray}
Here in second line above we used (\ref{C2-bound}), and in the third line we used (\ref{gradient-l}) and (\ref{convex-delta}).

Hence if  for some integer $j\ge -1$, $\delta'=2^{j+1}\delta$ and $\delta'/2\le \phi_0 \le \delta'$,  we estimate
\begin{eqnarray}
\nabla^2 \log \phi_0 (W, W) &=&\frac{\nabla^2 \phi_0(W, W)}{\phi_0}-\frac{|\nabla \phi_0|^2}{\phi_0^2}|W^\perp|^2\nonumber\\
&\le& -\frac{\theta_1\theta_2}{2\delta'}|W^{\operatorname{T}}|^2
+\frac{2C_1}{\delta'}|W^\perp|^2-\frac{\theta_1^2}{\delta'^2}|W^{\perp}|^2\nonumber \\
&\le&  -\frac{\theta_1\theta_2}{2\delta'}|W^{\operatorname{T}}|^2-\frac{\theta^2_1}{2(\delta')^2}|W^{\perp}|^2,\label{key}
\end{eqnarray}
for $\delta'\le \delta_1$ and $\gamma(s)\in \Omega\setminus \Omega_{\delta_1}$.
Here in the second line we  used (\ref{hessian-u1}) and in the third line we used the definition of $\delta_1$.

On the other hand, direct calculation shows
\begin{eqnarray}
\left(X(y)-X(x)\right)\cdot \frac{y-x}{|y-x|}&=& \langle X(\gamma(s)), \gamma'(s)\rangle |^{|y-x|}_0 \nonumber\\
&=& \int_0^{|y-x|}\frac{d}{ds} \left(\langle X(\gamma(s)), \gamma'(s)\rangle\right)\, ds\nonumber \\
&=& \int_0^{|y-x|} \nabla^2 (-\log \phi_0)(\gamma'(s), \gamma'(s))\, ds. \nonumber
\end{eqnarray}
Thus if $\gamma(s) \in \Omega\setminus \Omega_{\delta_1}$,  (\ref{key}) implies
\begin{equation}\label{1}
\left(X(y)-X(x)\right)\cdot \frac{y-x}{|y-x|}\ge 0.
\end{equation}
Otherwise, there exists $s''$, the first $s$ such that $\phi_0(\gamma(s))=\delta_1$.  Let $k$ be the integer such that $2^k \delta \le \delta_1< 2^{k+1} \delta$.
Since $|\nabla \phi_0|\le \frac{A}{2}$, we deduce that  if $s_j$ is the first $s$ with $\phi_0(\gamma(s))=\delta'/2$ and $s'_j$ is the first $s$ with $\phi_0(\gamma(s))=\delta'$, then $s'_j-s_j \ge \frac{\delta'}{A}$. Similarly $s''\ge \frac{\delta_1}{A}$, which particularly implies $|y-x|\ge \frac{\delta_1}{A}$. Clearly $s_{-1}=0, s'_j=s_{j+1}$. Now
\begin{eqnarray}
\left(X(y)-X(x)\right)\cdot \frac{y-x}{|y-x|}&\ge&\sum_{j=-1}^{k-1} \int_{s_j}^{s'_j} +\int_{s_k}^{s''}+\int_{s, \mbox{ with } \phi_0(\gamma(s))\ge \delta_1}\nabla^2 (-\log \phi_0)(\gamma', \gamma')\, ds \nonumber\\
&\ge& \frac{\theta^2_1\theta_3^2}{4\delta A}\sum_{j=-1}^{k-1}\frac{1}{2^j}-\frac{4 A^2}{\delta_1^2}D\nonumber\\
&\ge& \frac{C_2}{\delta}-\frac{C_3}{\delta_1^2}. \label{boundary-infty}
\end{eqnarray}
Here in the first line  (\ref{1}) is used, in the second line (\ref{key}), (\ref{acute}) are used, and that $C_2=\frac{\theta^2_1\theta_3^2}{2 A}$,  $C_3= 4A^2D$.

The estimates (\ref{1}), (\ref{boundary-infty}) together with the fact that $\psi$ is uniformly continuous on $\overline{\Omega}\times \overline{\Omega}$ implies the claim (\ref{claim}) hence that for $\delta$ small $\mathcal{C}(x, y)\ge -\epsilon$ on $\partial\left(\Omega_{\frac{\delta}{2}}\times \Omega_{\frac{\delta}{2}} \setminus \Delta_\eta\right)$. Taking $\delta\to 0$ we have that
$$
\mathcal{C}(x, y)\ge -\epsilon
$$
for $x, y\in \Omega$. Taking $\epsilon \to 0$ and then $D'\to D$ we have (\ref{AC-est1}).

\begin{remark} First observe that (\ref{boundary-infty}) implies that for $\delta $ sufficiently small, (\ref{1}) holds for any $(x, y)\in \partial(\Omega_{\frac{\delta}{2}}\times \Omega_{\frac{\delta}{2}})$. By taking $\psi(s)=-\epsilon \tan (\epsilon s)$ with $\epsilon>0$ small, Theorem \ref{max-e} and  estimate (\ref{1})  showed $\mathcal{C}(x, y)\ge 0$ for this case, which implies  Brascamp-Lieb's result on the log-concavity of $ \phi_0$ as $\epsilon \to 0$.\end{remark}

Replacing $\psi$ by $\tilde{\psi}=c\psi(cs)$ with $0<c<1$ and letting $c\to 1$, the same argument as above proves the following corollary which asserts that `boundary convexity'  implies  the `strong convexity' even the domain may not be convex.

\begin{corollary} Assume that $\Omega$ is a bounded domain such that there exists a smooth exhaustion $\Omega_\delta$ with $\Omega_\delta \to \Omega$ as $\delta\to 0$. Assume that $X$ be a $C^2(\Omega)$ vector field satisfying (\ref{eq-x}), such that there exists $\delta_1>0$,   for $(x, y)\in \partial(\Omega_\delta\times \Omega_\delta)$ with $|y-x|\le \delta_1$ (\ref{1}) holds, and for $(x, y)$ with $|x-y|\ge \delta_1$, $(X(y)-X(x))\cdot \frac{y-x}{|y-x|}\to +\infty$ as $\delta\to 0$. Let $\psi$ be as in Theorem \ref{max-e}.
Then $\mathcal{C}(x, y)\ge 0$ for any $x, y\in \Omega$.
\end{corollary}

\section{Improved log-concavity estimate  on the fundamental solution}

 Here we improve the log-concavity of the fundamental solution proved in \cite{BL}. Let $H(z, x, t)$ be the fundamental solution to the heat operator $\frac{\partial}{\partial t} -\mathcal{L}_q$ with Dirichlet boundary value on a strictly convex domain $D$. We use $K(z,x, t)$ to denote the Euclidean heat kernel $\frac{1}{(4\pi t)^{n/2}}\exp(-\frac{|x-z|^2}{4t})$. It was proved  in \cite{BL} that $\phi(x, t)\doteqdot\left(\frac{H}{K}\right)(z, x, t)$ is a log-concave function of $x$. The improved estimate asserts that $-\nabla \log \phi$ has a expansion modulus given by the one dimensional case. Before we state the improved version precisely,
  first let $\bar{H}(s, t)$ and $\bar{K}(s, t)$ be the corresponding  fundamental solutions (concentrated at $s=0$) on $[-\frac{D}{2}, \frac{D}{2}]$  and $\mathbb R$ for operator $\frac{\partial}{\partial t}-\frac{\partial^2}{\partial s^2}$. Let $\psi(s, t)=\left(\log \left(\frac{\bar{H}}{\bar{K}}\right)\right)'$.

 \begin{theorem}\label{AC-p1} With the notation above, if $q$ is convex, then for any $t\ge 0$
 \begin{equation}\label{AC-est-p2}
 -\left(\nabla_y \log \phi (y, t)-\nabla_x\log \phi(x, t)\right)\cdot \frac{y-x}{|y-x|}\ge -2\psi\left(\frac{|y-x|}{2}, t\right).
 \end{equation}
 \end{theorem}

 The estimate (\ref{AC-est-p2}) has the following equivalent form:
 \begin{equation}\label{AC-est-p3}
 \left(\nabla_y \log H (z, y, t)-\nabla_x\log H(z, x, t)\right)\cdot \frac{y-x}{|y-x|}\le 2\left(\log \bar{H}\right)'\left(\frac{|y-x|}{2}, t\right).
 \end{equation}
 Since $(\log\bar{H})'<0$ and $\left(\log \left(\frac{\bar{H}}{\bar{K}}\right)\right)'<0$ on $(0, \frac{D}{2})$, Theorem \ref{AC-p1} and (\ref{AC-est-p3}) improve the earlier result of Brascamp-Lieb. It is sharp since the equality holds for dimension one.
 For its proof we need the following variation  of Theorem \ref{max-p}.

 \begin{theorem}\label{max-p1} Let $\Omega$ be  bounded domain of $\mathbb R^n$ with diameter $D$. Let $X(x, t)$ be a $C^2$-vector field defined on $\Omega\times(0, T]$ satisfying the equation
 \begin{equation}\label{eq-x-p2}
 \left(\frac{\partial}{\partial t} -\Delta\right) X(x, t) =V(x, X(x, t))-2\nabla_X X(x, t) -\frac{1}{t}\langle \nabla_{(\cdot)} X(x, t), x-z\rangle -\frac{1}{t}X(x, t)
 \end{equation}
 where $z\in \mathbb R^n$ is fixed, with $V(x, X)$ being  jointly convex.  Assume further that $X$ is symmetric i.e. $\langle \nabla _{W_1} X, W_2\rangle =\langle \nabla_{W_2} X, W_1\rangle$ for any $W_i$.  Let $\psi(s, t)$ be as in Theorem \ref{AC-p1}, or more generally a $C^{1, 2}$-function on $[0, \frac{D}{2})\times \mathbb{R}\to \mathbb R$ with $\psi(0, t)=0$, $\psi'(s,t)<0$ for $s>0$, and satisfying
 \begin{equation}\label{psi-inq}
\psi_t-\psi''\ge 2\psi \psi' -\frac{\psi}{t}-\frac{\psi'}{t}s.
 \end{equation}
  Then
 $$\mathcal{C}(x, y, t)\doteqdot t\left(\langle X(y, t)-X(x, t), \frac{y-x}{|y-x|}\rangle +2\psi\left(\frac{|y-x|}{2}, t\right)\right)$$
 can not attain a negative minimum in the parabolic interior.
 \end{theorem}

Here $\langle \nabla_{(\cdot)} X, x-z\rangle$ is a vector whose inner product with any vector $W$ is $\langle \nabla_W X, x-z\rangle$.
Direct calculation shows that $X(x, t)\doteqdot -\nabla \log \phi(x, t)$ satisfies the equation (\ref{eq-x-p2}) with $V(x, X)=\nabla q(x)$. We first prove Theorem \ref{max-p1}.

\begin{proof} Argue by contradiction. Assume that at $(x_0, y_0, t)$ with $t>0, x_0, y_0\in \Omega$, $\mathcal{C}$ attains a negative minimum on $\Omega \times \Omega \times (0, T]$. Following the notations from the proof to Theorem \ref{max-e}, the first variation consideration yields (\ref{1st-eq1}) and (\ref{1st-eq2}), which together imply that
\begin{equation}\label{1st-eq-dual}
\langle \nabla_{(\cdot)} X(y),  e_n\rangle =\langle \nabla_{(\cdot)} X(x), e_n\rangle.
\end{equation}
From now on, in the proof,  when the meaning is clear we omit $t$ variable dependence in  $X(x, t)$.
Now we compute
\begin{eqnarray*}
0&\ge& \left(\frac{\partial}{\partial t}-\sum_{j=1}^n \nabla^2_{E_j E_j}\right) \mathcal{C}(x, y, t)|_{(x_0, y_0, t)}\\
&=&t\langle V(y, X(y))-V(x, X(x)), \frac{y-x}{|y-x|}\rangle -2t\langle \nabla_{X(y)}X(y)-\nabla_{X(x)}X(x), \frac{y-x}{|y-x|}\rangle \\
&\,& -\langle \nabla_{e_n} X(y), y-z\rangle +\langle \nabla_{e_n} X(x), x-z\rangle -\langle X(y)-X(x), \frac{y-x}{|y-x|}\rangle\\
&\,& +2t(\psi_t -\psi'')+\langle X(y)-X(x), \frac{y-x}{|y-x|}\rangle+2\psi.
\end{eqnarray*}
Here the right hand side is evaluated at $(x_0, y_0)$ and we have used (\ref{eq-x-p2}). The first term is nonnegative by the convexity assumption on $V(x, X)$. As in the proof of Theorem \ref{max-e}, the equation (\ref{1st-eq1}) and (\ref{1st-eq2}) implies that the second term equals
\begin{equation}\label{help-thm61-1}
\frac{2t}{|y-x|}\sum_{i=1}^{n-1}\langle X(y)-X(x), e_i\rangle^2+2t\psi' \langle X(y)-X(x), e_n\rangle.
\end{equation}
Applying (\ref{1st-eq2}) again, at $(x_0, y_0, t)$, we have
\begin{equation}\label{help-thm61-2}
-\langle \nabla_{e_n} X(y), y-z\rangle +\langle \nabla_{e_n} X(x), x-z\rangle=\psi' |y-x|.
\end{equation}
Combining  the previous computation with (\ref{help-thm61-1}) and (\ref{help-thm61-2}) we have that,  at $(x_0, y_0, t)$,
\begin{eqnarray}
0&\ge& 2t\psi' \langle X(y)-X(x), e_n\rangle+\psi' |y-x|+2\psi +2t(\psi_t-\psi''). \label{almost1}
\end{eqnarray}
On the other hand $\psi$ satisfies (\ref{psi-inq}).
Plugging (\ref{psi-inq}) and  $s=\frac{|y-x|}{2}$ into it, (\ref{almost1}) then implies at $(x_0, y_0, t)$
\begin{eqnarray*}
0&\ge & 2t\psi' \langle X(y)-X(x), e_n\rangle +4t\psi' \psi\\
&=&2\psi' \mathcal{C}.
\end{eqnarray*}
This is a contradiction to $\mathcal{C}(x_0, y_0, t)<0$ and the fact that $\psi'<0$ (which follows from the log-concavity of $\frac{\bar{H}}{\bar{K}}$ and the strong maximum principle, noticing $x_0\ne y_0$).
\end{proof}

 Observe that $\phi(x, t)=\frac{H(z, x,t)}{K(z, x, t)}$ here also satisfies that $\phi(x, t)>0$ on $\Omega$, $\phi(x, t)=0$ on $\partial \Omega$ and the partial differential equation:
$$
\left(\frac{\partial}{\partial t} -\Delta\right)\phi =-q\phi +2\langle \nabla \phi, \nabla \log K\rangle.
$$
Hence the parabolic Hopf's lemma implies that $\frac{\partial \phi}{\partial \nu} <0$ on $\partial \Omega$. Therefore the same argument of Section 3 implies the estimates (\ref{1}) and (\ref{boundary-infty}) for $X(x, t)=-\nabla \log \phi(x, t)$ on the  points near the boundary.
Now  replacing $\psi$ with $\widetilde \psi(s, t)=\epsilon \psi(\epsilon s, \epsilon^2 t)$ with $\epsilon \in (0, 1)$, and  observing that the heat kernel asymptotics (cf. \cite{strook}, \cite{Neel})  imply that $\mathcal{C}(x, y, t)\ge 0$ holds  at $t=0$, Theorem \ref{max-p1} implies that $\mathcal{C}(x, y, t)\ge 0$. Letting $\epsilon \to 0$ we get a maximum principle proof for Brascamp-Lieb's log-concavity of $\frac{H}{K}$ and letting $\epsilon \to 1$ we get  Theorem \ref{AC-p1}. The general $\epsilon$ serves a natural interpolation between the strong and the weak result.

\begin{remark} By taking $t\to \infty$, since $e^{\lambda_0t}H(z, x, t)\to \phi_0(z) \phi_0(x)$, the estimate (\ref{AC-est-p3}) implies the improved log-concavity estimate (\ref{AC-est1}).
One can formulate a general maximum principle
for the case that $V(x, X)$ has a convexity module as in Theorem \ref{max-p}. Similarly,  Theorem \ref{AC-p1} can be generalized to the  case that $\nabla q$ has an expansion modulus $\omega(s)$ as in \cite{AC}. Without insisting $\psi'<0$ in Theorem \ref{max-p1}, the argument also proves that $\mathcal{C}(x, y, t)\ge 0$ is preserved by (\ref{eq-x-p2}).
\end{remark}

\section{Alternate proof to the fundamental gap theorem}

In \cite{AC}, by relating the fundamental gap to the exponential decay rate of the solution to a parabolic equation, the authors proved the following result.
\begin{theorem}[Andrews-Clutterbuck] \label{AC-gap} Let $\Omega$ be a strictly convex bounded domain in $\mathbb R^n$ with diameter $D$. Then the gap between the second eigenvalue $\lambda_1$ and the first $\lambda_0$ (for the operator $\mathcal{L}_q$ with $q$ being convex) satisfies:
\begin{equation}\label{gap1}
\lambda_1-\lambda_0\ge \frac{3 \pi^2}{D^2}.
\end{equation}
\end{theorem}

Here we shall give an alternate argument without appealing to the parabolic equation. First recall (cf. \cite{SWYY}) that $w=\frac{\phi_1}{\phi_0}$ with $\phi_1$ being the eigenfunction corresponding to $\lambda_1$, is a $C^2(\overline{\Omega})$ function,  satisfying
\begin{equation}\label{eq-q}
\Delta w=-(\lambda_1-\lambda_0)w+2\langle \nabla w, X\rangle, \mbox{ in } \Omega\quad \mbox{ and } \frac{\partial w}{\partial\nu}=0, \mbox{ on } \partial \Omega.
\end{equation}
Here $X(x)=-\nabla \log \phi_0$.
Correspondingly, let $\mu_1=\frac{4\pi^2}{D^2}$ and $\mu_0=\frac{\pi^2}{D^2}$ be the first and second (Dirichlet) eigenvalues for the operator $\frac{d^2}{ds^2}$ on $[-\frac{D}{2}, \frac{D}{2}]$, and let $\bar{w}=\frac{\bar{\phi}_1}{\bar{\phi}_0}$ with $\bar{\phi}_1=\sin\left(\frac{2\pi}{D}s\right)$ and $\bar{\phi}_0=\cos\left(\frac{\pi}{D}s\right)$ being the eigenfunctions respectively. Note $\bar{w}$ satisfies (\ref{eq-q}) for $n=1$. We allow $D$ being replaced by $D'>D$ and denote the corresponding function by $\bar{\phi}_0^{D'}$. Clearly $\mu_1-\mu_0$ depends on $D'$ continuously and decreases as $D'$ increases.

Since $w$ is not a constant, $\bar{w}(s)$ is comparable with $s$ on $[0, \frac{D'}{2}]$, by the continuity it is easy to see that given small $\epsilon$ with $\frac{D'-D}{2}> \epsilon>0$ one may find positive constant $C$  such that
$$
\mathcal{O}(x, y)\doteqdot w(y)-w(x)-C\bar{w}\left(\frac{|y-x|}{2}+\epsilon\right)
$$
attains its maximum $0$ somewhere in $\overline{\Omega}\times \overline{\Omega}$. The strict convexity of $\Omega$,  the Neumann boundary condition on $w$ and the positivity of $\bar{w}'$ on $[0,\frac{D}{2}+\epsilon]$ rule out the possibility that the maximum is attained on $\partial\left(\Omega\times \Omega\right)$. So the maximum is attained for some interior point $(x_0, y_0)$. Clearly $x_0\ne y_0$ and $\bar{w}\ne 0$ at $(x_0, y_0)$. As \cite{AC} we pick a normal frame $\{e_i\}$ such that $e_n=\frac{y-x}{|y-x|}$ near $x$ and parallel translate them to $y$. Let $E_i=e_i\oplus e_i$ and $E_n=e_n\oplus (-e_n)$. The first variation of $\mathcal{O}(x,y)$ at $(x_0, y_0)$ asserts that
\begin{eqnarray}
&\,&\nabla_{e_i} w(y)=\nabla_{e_i} w(x)=0, \label{1stvar-1}\\
&\,& \nabla_{e_n} w(y)=\frac{C\bar{w}'}{2} =\nabla_{e_n} w(x)\label{1stvar-2}
\end{eqnarray}
which together imply
\begin{equation}\label{1stvar}
\nabla w(y)=\nabla w(x)=\frac{C\bar{w}'}{2} \frac{y-x}{|y-x|}.
\end{equation}
Now the second derivative test asserts that $\sum_{i=1}^n \nabla^2_{E_j E_j}\mathcal{O}(x, y) \le 0$, with
\begin{eqnarray}
\sum_{i=1}^{n-1}\nabla^2_{E_i E_i} \mathcal{O}(x, y)&=& \sum_{i=1}^{n-1}\nabla^2_{e_i e_i}w(y)-\sum_{i=1}^{n-1}\nabla^2_{e_i e_i}w(x),\label{2ndvar-1}\\
\nabla^2_{E_n E_n} \mathcal{O}(x, y) &=& \nabla^2_{e_n e_n}w(y)-\nabla^2_{e_n e_n} w(x) -C\bar{w}''. \label{2ndvar-2}
\end{eqnarray}
 Putting (\ref{2ndvar-1}), (\ref{2ndvar-2}), (\ref{1stvar}) and (\ref{eq-q}) together we have that at $(x_0, y_0)$,
 \begin{eqnarray*}
 0&\ge& \sum_{i=1}^n \nabla^2_{E_j E_j}\mathcal{O}(x, y)\\
 &=& -(\lambda_1-\lambda_0)(w(y)-w(x))+(C \bar{w}')\langle X(y)-X(x), \frac{y-x}{|y-x|}\rangle -C\bar{w}'' \\
 &\ge& -(\lambda_1-\lambda_0)(w(y)-w(x))+2(C \bar{w}')\left(-\log \bar{\phi}_0^{D'}\right)'-C \bar{w}''\\
 &\ge&-(\lambda_1-\lambda_0)(w(y)-w(x))+(\mu_1-\mu_0)C \bar{w}.
 \end{eqnarray*}
 Here we have used (\ref{AC-est1}), $\bar{w}'\ge0$ and $-(\log\bar{\phi}^{D'}_0)'$ is a decreasing function of $D'$. This implies that $\lambda_1-\lambda_0\ge \mu_1-\mu_0$ for any $D'>D$. Let $D'\to D$, we get the result.

 We remark that in the proof above $\phi_1$ can be replaced by any eigenfunction $\phi_i$ with $i\ge 1$. But we made use that $\frac{\bar{\phi}_1}{\bar{\phi}_0}>0$ for $s\in [ \epsilon, \frac{D'}{2}]$.

    We refer the interested readers to \cite{AC} for the motivation, history and comprehensive literatures on previous works related to Theorem \ref{AC-gap}.

 \section{Further applications of the improved log-concavity estimate}

 First the argument of the last section effectively proves the following result on a lower bound of the second Neumann eigenvalue for the operator$\Delta -2\langle \nabla (\cdot) , X\rangle$ (with nonconstant eigenfunction).

\begin{corollary}\label{Neumann} (i) If $X$ has an expansion modulus given by $-(\log \bar{\phi}_0)'$, then the second Neumann eigenvalue $\widetilde{\lambda}_1$
of the operator $\Delta -2\langle \nabla (\cdot) , X\rangle$ is bounded from below by $\mu_1-\mu_0$;

(ii) If $X$ is merely convex, or more generally $X$ has $\epsilon' \frac{\bar{w}}{\bar{w}'}$ as its expansion modulus for  $\epsilon'> -\frac{\mu_0}{2}$,  then $\widetilde{\lambda}_1\ge 2\epsilon'+ \mu_0$. The convexity of $X$ amounts to $\epsilon=0$.

Both results still hold for $\Omega$ being a convex domain in a Riemannian manifold with nonnegative Ricci curvature, or for any compact Riemannian manifold (without boundary)  with nonnegative Ricci curvature.
 \end{corollary}

Recall that $\mu_0=\left(\frac{\pi}{D}\right)^2$ and $\bar{\omega}=2\sin(\frac{\pi}{D}s)$. Part (i) is obvious. For the proof of the second statement in Corollary \ref{Neumann}, letting $w$ be the first non-trivial eigenfunction in the argument of the last section, it suffices to observe that $\bar{w}''=-\mu_0 \bar{w}$. The part (ii)  generalizes  an earlier  result of Payne-Weinberger \cite{PW} which asserts the same statement for $X(x)\equiv 0$. Note that it even  applies to  the case that $\epsilon<0$.
One candidate of the vector field $X$ satisfying the assumption of the part (i) is  $-\nabla \log \widetilde{\phi}_0$ with $\widetilde{\phi}_0$ being  the first eigenfunction of some domain $\Omega'$ containing $\Omega$, but with the same diameter. For the last statement, after some obvious modifications on the definition of the modulus of expansion and replacing $|y-x|$ by $r(x, y)$ (the distance function), it suffices to observe that the second variation (without fixing either end) of the distance function $r(x, y)$ is non-positive while $\bar{w}'\ge 0$. Hence the proof goes without any changes. Note that on a compact Riemannian manifold any convex vector field $X$ (being convex is equivalent to that $\langle \nabla_W X, W\rangle \ge 0$) must be parallel. Hence statement here generalizes a corresponding  result of Li-Yau \cite{LY} and Zhong-Yang \cite{ZY} for convex domains in a Riemannian manifold (or for a compact Riemannian manifold when $X$ is non-convex). This result can also be derived from Theorem 2.1 of \cite{AC} by the consideration in Section 3 of that paper.

Secondly we consider the case that $\Omega$ be a compact manifold with $\operatorname{Ric}\ge n-1$ or a bounded convex domain in such a Riemannian manifold.  The argument in the last section can yield an interpolating  estimate on $\tilde{\lambda}_1$.  First recall the following lemma which may be well known for experts.

\begin{lemma} Assume that $x, y\in M$ with $\operatorname{Ric}\ge n-1$. Let $\gamma(s)$ be a minimizing geodesic joining $x$ and $y$. Let $\{e_i\}$ be a orthonormal frame at $x$ and parallel translate it along $\gamma(s)$ with $e_n=\gamma'(s)$. Then for $x, y\in M$, with distance $r(x, y)< \pi$,
\begin{equation}\label{dis-var}
\sum_{i=1}^{n-1} \nabla^2_{E_i E_i} r(x, y) \le -2(n-1)\frac{\sin \left(\frac{r(x, y)}{2}\right)}{\cos \left(\frac{r(x, y)}{2}\right)}.
\end{equation}
\end{lemma}
\begin{proof} Since the distance function $r(x, y)$ may not be smooth, the estimate is understood in the sense of support. Let $\gamma_i(s, \eta)=\exp_{\gamma(s)} (\eta V_i(s))$ for $i=1, \cdots, n-1$ with $V_i=\left(\cos(s)+\frac{\cos d-1}{\sin d}\sin s \right)e_i(s)$. Here we denote $r(x, y)$ by $d$. Since $\frac{D \gamma}{\partial \eta}(0, \eta)=e_i(0), \frac{D \gamma}{\partial \eta}(d, \eta)=e_i(d)$ and $r(\gamma(0, \eta), \gamma(d, \eta)\le L(\gamma(s, \eta))$, the arc-length of $\gamma(\cdot, \eta)$,    the second variation formula implies the following differential inequality in the barrier sense,
\begin{eqnarray*}
\nabla^2_{E_i E_i} r(x, y)&\le&\left. \frac{d^2}{ d\eta^2} \int_0^d \left|\frac{D \gamma}{\partial s}\right|(s, \eta)\, \, ds\right|_{\eta=0}\\
&=& \int_0^d \left(|V_i'|^2 -\langle R( V_i, \gamma') V_i, \gamma'\rangle\right)\, ds.
\end{eqnarray*}
The lemma follows by summing the above for $i=1$ to $n-1$, plugging in the assumption $\operatorname{Ric}\ge n-1$,  and elementary identities.
\end{proof}

The argument in the last section then shows the following result regarding the second Neumann eigenvalue of $\Delta-2\langle \nabla (\cdot) , X\rangle$.
\begin{corollary}\label{Li-Li} Let $\Omega$ be a compact manifold, or a convex domain in  a Riemannian manifold, with $\operatorname{Ric}\ge (n-1)K$. Assume that the diameter $D<\frac{\pi}{\sqrt{K}}$, and that $X$ has a modulus of expansion $\epsilon' \frac{\bar{w}}{\bar{w}'}$. Then for any $D'\in( D, \frac{\pi}{\sqrt{K}}]$,
$$\widetilde{\lambda}_1\ge 2\epsilon' + (n-1)\frac{\pi}{D'} \inf_{0\le r\le D}\frac{\tan (\sqrt{K}\frac r 2)}{\tan (\frac{\pi}{D'} \frac r 2)} +\left(\frac{\pi}{D'}\right)^2.$$
\end{corollary}

We should remark that a similar, seemingly more geometrically formulated result, but with $X=0$, was obtained by Andrews and Clutterbuck as described in \cite{A}. The formulation here is a bit simple-minded. Nevertheless it gives an explicit lower bound, and taking $D'=\pi$, the result contains the Lichnerowicz's $\widetilde{\lambda}_1\ge n$ (so does the formulation of Andrews-Clutterbuck), and when $K=0$ it recovers Li-Yau, Zhong-Yang's estimate. Hence it addresses a conjecture of P. Li \cite{Ling}.

\begin{proof} It suffices to prove for $K=1$. The only difference is on (\ref{2ndvar-1}), which now becomes
$$
\sum_{i=1}^{n-1}\nabla^2_{E_i E_i} \mathcal{O}(x, y)\ge \sum_{i=1}^{n-1}\nabla^2_{e_i e_i}w(y)-\sum_{i=1}^{n-1}\nabla^2_{e_i e_i}w(x)+(n-1) C\bar{w}' \tan \left(\frac{r(x, y)}{2}\right).
$$
Then the result of argument in the alternate proof of Theorem \ref{AC-gap} of the previous section  shows that
$$
\widetilde{\lambda}_1\ge 2\epsilon'+\left(\frac{\pi}{D'}\right)^2+(n-1)\frac{\pi}{D'}\frac{\tan(\frac{r_0}{2})}{\tan(\frac{\pi}{D'}\frac{r_0+\epsilon}{2})}
$$
with $r_0=r(x_0, y_0)$. Taking $\epsilon \to 0$ the claimed result then follows.\end{proof}

 The estimate (\ref{AC-est1}) has  another application on the lower estimate of $\lambda_0$, the first (Dirichlet) eigenvalue of the operator $\mathcal{L}_q$.

 \begin{corollary}\label{lambda0} Assume that $\Omega$ is a bounded convex domain in $\mathbb R^n$ with diameter $D$. Assume that $q(x)$ is convex.
Then
\begin{equation}
\lambda_0\ge n\left(\frac{\pi}{D}\right)^2 +\inf_{x\in \Omega} q(x).
\end{equation}
It then implies that the second eigenvalue has the lower bound estimate:
\begin{equation}
\lambda_1\ge (n+3)\left(\frac{\pi}{D}\right)^2 +\inf_{x\in \Omega} q(x).
\end{equation}
 \end{corollary}
 \begin{proof} Since $\phi_0=0$ on $\partial \Omega$ and $\phi_0>0$, it must attain its maximum for some $x_0\in \Omega$. For $r$ small integrate the estimate (\ref{AC-est1}) over the $\partial B_{x_0}(r)$:
 \begin{eqnarray*}
 \omega_{n-1} r^{n-1}2\frac{\pi}{D}\tan \left(\frac{\pi r}{2D}\right) &\le &\int_{\partial B_{x_0}(r)} \left(X(y)-X(x_0)\right)\cdot \nu \, dA(y)\\
 &=&\int_{B_{x_0}(r)}\operatorname{div} X(y)\, d\mu(y)\\
 &=& \lambda_0 \frac{\omega_{n-1}}{n} r^n +\int_{B_{x_0}(r)} \left(|\nabla \log \phi_0|^2 -q\right)\, d\mu(y).
 \end{eqnarray*}
 Here $\omega_{n-1}$ is the area of the $\partial B_0(1)$ and recall that $X=-\nabla \log \phi_0$. Hence we have that
 \begin{eqnarray*}
 \lambda_0 &\ge& \frac{2n\pi}{r D} \tan \left(\frac{\pi r}{2D}\right)+\aint_{B_{x_0}(r)}\left(q-|\nabla \log \phi_0|^2 \right)\, d\mu(y).
 \end{eqnarray*}
 Taking $r\to 0$ in the right hand side above we get the desired result, since $\nabla\phi_0(x_0)=0$ and
 $$
 \lim_{r\to 0}\frac{2n\pi}{r D} \tan \left(\frac{\pi r}{2D}\right)=\frac{n\pi^2}{D^2}.
 $$
 \end{proof}

 Using the isodiametric inequality, Corollary \ref{lambda0} implies that
 $$
 \lambda_0\ge n \frac{\pi^2}{4}\left(\frac{\alpha(n)}{|\Omega|}\right)^{2/n}+\inf_{\Omega} q, \quad \lambda_1 \ge (n+3)\frac{\pi^2}{4}\left(\frac{\alpha(n)}{|\Omega|}\right)^{2/n}+\inf_{\Omega} q.
 $$
 Here $\alpha(n)$ is the volume of the unit ball in $\mathbb R^n$. Note that for $v=0$, the Polya's conjecture  asserts the lower bound $\lambda_{j}\ge 4\pi^2 \left(\frac{j+1}{\alpha(n) |\Omega|}\right)^{2/n}$. The estimate for $j=0, 1$ given above are better than the conjectured lower bounds,  however since it was motivated by Weyl's asymptotics, the main interesting cases are for  big $j$.

Using a similar argument as in the proof of Corollary \ref{lambda0}, (\ref{BL-imp}) implies the following sharp upper bound on the growth rate of $H(x, y, t)$.

\begin{corollary}\label{con-heat} Assume that $\Omega$ is convex and $q(x)$ is convex. Let $H(z, x, t)$ be the Dirichlet heat kernel for $\frac{\partial}{\partial t}-\mathcal{L}_q$ with potential function $q$. For any fixed $z\in \Omega$, let $m(z, t)\doteqdot \max_{x\in \Omega} H(z, x, t)$. Then
$$
\frac{d}{dt}\log  m(z, t)\le n \frac{d}{dt} \log \bar{H}(0, t)-\inf q.
$$
\end{corollary}
\begin{proof} Since $m(y, t)$ may not be smooth in general, the derivative is understood as the Dini derivative from the left. Since $H(z,x, t)$ takes the $0$ value on the boundary, it attains its maximum interior. Let $x(t)$ be such a point where  $m(z, t)$ is attained. Then
\begin{eqnarray*}
\frac{d}{dt} \log m(z, t)&\le & \lim_{h\to 0}\frac{\log H(z, x(t), t)-\log H(z, x(t), t-h)}{h}\\
&\le & \frac{\Delta_{x} H(z, x(t), t)}{H(z, x(t), t)}-\inf q\\
&=& \Delta \log H(z, x(t), t)-\inf q\\
&\le & n\lim_{s\to 0}\frac{(\log \bar{H}(s, t))'}{s}-\inf q\\
&\le & n \left(\log \bar{H}(0, t)\right)''-\inf q\\
&=&  n \left(\log \bar{H}\right)_t (0, t)-\inf q.
\end{eqnarray*}
Here in the third equation $\nabla H(z, x(t), t)=0$ is used;   in line 4 estimate (\ref{BL-imp}) is used; in the last line the fact that $\bar{H}'(0, t)=0$ is used.
\end{proof}

Note that Corollary \ref{con-heat} implies Corollary \ref{lambda0} since the decay rate of $H(z, x, t)$ is $e^{-\lambda_0 t}$ and the decay rate of $\bar{H}(0, t)$ is $e^{-\mu_0t}$.

\section*{Acknowledgments.} {  }
We thank Ben Andrews for the communications regarding Corollary 4.4 of \cite{AC} and his interests, Peter Li for conversations regarding his conjecture, Bruce Driver for discussions. The graduate course taught by the author at UCSD in Spring 2011 is the main motivation of the proofs in Sections 3 and 5. We also thank Alexander Grigoryan for reference \cite{strook}.

\bibliographystyle{amsalpha}

\end{document}